\documentclass[12pt,leqno,a4paper]{article}
\usepackage{amsthm}
\usepackage{indentfirst}
\usepackage{amsmath}
\usepackage{mathrsfs}
\usepackage{paralist}
\usepackage{amssymb,enumerate}
\usepackage{hyperref}
\usepackage{framed}
\usepackage{extarrows}
\usepackage[hmargin=3cm,vmargin=3cm]{geometry}
\usepackage{url}
\usepackage{color}

\theoremstyle{theorem}
\newtheorem{mainthm}{Theorem}

\theoremstyle{theorem}
\newtheorem{thm}{Theorem}[section]
\newtheorem{prop}[thm]{Proposition}
\newtheorem{lem}[thm]{Lemma}

\newtheorem{exmp}[thm]{Example}

\theoremstyle{definition}

\newtheorem{rem}[thm]{Remark}

\newcommand{\Cl}{\operatorname{Cl}}
\newcommand{\GL}{\operatorname{GL}}
\newcommand{\GU}{\operatorname{GU}}

\newcommand{\Irr}{\operatorname{Irr}}
\newcommand{\Ker}{\operatorname{Ker}}
\newcommand{\Lin}{\operatorname{Lin}}
\newcommand{\PGL}{\operatorname{PGL}}
\newcommand{\PSL}{\operatorname{PSL}}
\newcommand{\PSU}{\operatorname{PSU}}
\newcommand{\Res}{\operatorname{Res}}
\newcommand{\SL}{\operatorname{SL}}
\newcommand{\SO}{\operatorname{SO}}
\newcommand{\Sp}{\operatorname{Sp}}
\newcommand{\SU}{\operatorname{SU}}

\newcommand{\bG}{\mathbf{G}}
\newcommand{\tbG}{\tilde{\mathbf{G}}}

\newcommand{\bT}{\mathbf{T}}
\newcommand{\bS}{\mathbf{S}}

\newcommand{\bbG}{\mathbb{G}}
\newcommand{\bbF}{\mathbb{F}}
\newcommand{\bbN}{\mathbb{N}}

\newcommand{\bbZ}{\mathbb{Z}}
\newcommand{\barF}{\overline{\mathbb{F}}}

\newcommand{\cE}{\mathcal{E}}
\newcommand{\cF}{\mathcal{F}}
\newcommand{\cP}{\mathcal{P}}
\newcommand{\cZ}{\mathcal{Z}}

\newcommand{\hz}{\hat{z}}
\newcommand{\ts}{\tilde{s}}
\newcommand{\tDelta}{\tilde{\Delta}}
\newcommand{\tcF}{\widetilde{\mathcal{F}}}
\newcommand{\grp}[1]{\langle#1\rangle}

\newcommand{\set}[1]{\{\,#1\,\}}
\newcommand{\Set}[1]{\left\{\,#1\,\right\}}
\newcommand{\lc}[1]{{^{#1}\!}}

\title{Restrictions of irreducible characters of finite groups of Lie type to derived subgroups and regular semisimple elements
\footnote{The author gratefully acknowledges financial support by NSFC (No. 11901478, No. 11631001).
\newline \textbf{2010 Mathematics Subject Classification:} 20C15, 20C33.
\newline \textbf{Keyword:} irreducible characters, finite groups of Lie type, restrictions, derived subgroups, strongly regular semisimple elements}}
\author{Conghui Li\\
School of Mathematics, Southwest Jiaotong University,\\
Chengdu 611756, China\\
Email: liconghui@swjtu.edu.cn}
\date{}

\begin{document}

\setlength\abovedisplayskip{1ex plus 0.2ex minus 0.1ex}
\setlength\belowdisplayskip{1ex plus 0.2ex minus 0.1ex}

\raggedbottom
\maketitle

\begin{abstract}
In this note, we formulate an observation that ``almost all'' irreducible ordinary characters of finite groups of Lie type remain irreducible when restricted to the derived subgroups.
To see this, key ingredients are some asymptotic results for conjugacy classes of finite groups of Lie type and strongly regular semisimple elements in dual groups.
\end{abstract}

\section{Introduction}

In this note, assume $q$ is a power of a prime $p$.
Denote by $\bbF_q$ the finite field of $q$ elements and by $\barF_q$ the algebraic closure of $\bbF_q$.

By a glance at the character tables of $\GL_2(q)$, $\SL_2(q)$, $\GL_3(q)$, $\SL_3(q)$ (see \cite{DM20}, \cite{St51}, \cite{FS73}), a phenomenon can be noticed that ``almost all'' irreducible characters of $\GL_n(q)$ remain irreducible when restricted to $\SL_n(q)=[\GL_n(q),\GL_n(q)]$ for $n=2,3$.

The above observation can be formulated in general.
We first give some notation.
For any finite group $G$, denote by $\Irr(G)$ the set of all irreducible complex characters of $G$,
and denote by $\Irr_{i.r.d.}(G)$ the set of all irreducible complex characters of $G$ remaining irreducible when restricted to the derived subgroups $[G,G]$.
Let $\bbG=(X,R,Y,R^\vee,W\phi)$ be a generic group (see \cite[\S22.2]{MT11} for example).
Denote by $(\bG,\bT,F)$ the triple determined by $\bbG$ and $q$ consisting of a connected reductive group $\bG$ over $\barF_q$, a Frobenius map $F$ of $\bG$ defining an $\bbF_q$-structure on $\bG$ and $\bT$ an $F$-stable maximally split maximal torus of $\bG$.
Here, there is no loss to exclude Suzuki and Ree groups and to consider only Frobenius maps.
Then $\bbG(q)=\bG^F$ is the finite group of Lie type determined by $\bbG$ and $q$.
The above observation can be formulated as follows.

\begin{mainthm}\label{mainthm}
Let $\bbG=(X,R,Y,R^\vee,W\phi)$ be a generic group.
Denote by $(\bG,\bT,F)$ the triple determined by $\bbG$ and $q$, and set $\bbG(q)=\bG^F$.
Then we have that
\[ \lim_{q\to\infty}\frac{|\Irr_{i.r.d.}(\bbG(q))|}{|\Irr(\bbG(q))|}=1. \]
\end{mainthm}

A key ingredient is an asymptotic result for strongly regular semisimple elements.
For any finite group $G$, denote by $\Cl(G)$ the set of conjugacy classes of $G$.
Let $\bG$ be a connected reductive group.
Recall that (\cite{St65}) a semisimple element $s$ of $\bG$ is (strongly) regular if $C_{\bG}(s)^\circ$ ($C_{\bG}(s)$) is a maximal torus of $\bG$.
For a connected reductive group or a finite group of Lie type $G$,
denote by $\Cl(G)_s$, $\Cl(G)_{r.s.}$ and $\Cl(G)_{s.r.s}$ the set of conjugacy classes of semisimple, regular semisimple and strongly regular semisimple elements of $G$ respectively.

\begin{mainthm}\label{mainthm-srs}
Let $\bbG=(X,R,Y,R^\vee,W\phi)$ be a generic group and denote by $l$ the rank of the root datum $(X,R,Y,R^\vee)$.
Denote by $(\bG,\bT,F)$ the triple determined by $\bbG$ and $q$, and set $\bbG(q)=\bG^F$.
Then we have that
\[ \lim_{q\to\infty}\frac{|Z(\bG)^{\circ F}|q^l}{|\Cl(\bbG(q))_{s.r.s}|}
=\lim_{q\to\infty}\frac{|Z(\bG)^{\circ F}|q^l}{|\Cl(\bbG(q))_{r.s.}|}
=\lim_{q\to\infty}\frac{|Z(\bG)^{\circ F}|q^l}{|\Cl(\bbG(q))_s|}=1. \]
\end{mainthm}

Next is an asymptotic result for numbers of conjugacy classes of finite groups of Lie type.

\begin{mainthm}\label{mainthm-cl}
Let $\bbG=(X,R,Y,R^\vee,W\phi)$ be a generic group and denote by $l$ the rank of the root datum $(X,R,Y,R^\vee)$.
Denote by $(\bG,\bT,F)$ the triple determined by $\bbG$ and $q$, and set $\bbG(q)=\bG^F$.
Then we have that
\[ \lim_{q\to\infty}\frac{|Z(\bG)^{\circ F}|q^l}{|\Cl(\bbG(q))|}=1. \]
\end{mainthm}

The result of the above theorem for some classical groups has been included in \cite{FG12}.
An intuitive explanation of the above theorems is a result in \cite{St65} claiming that the strongly regular semisimple elements form a dense set in $\bG$.

\paragraph{Acknowledgement}
I am extremely grateful to Professor Meinolf Geck for the suggestion to consider strongly regular semisimple elements.

\section{Proofs}\label{proofs}

For a connected reductive group $\bG$ with a Frobenius map $F$,
we denote by $\Cl(\bG)_s^F$, $\Cl(\bG)_{r.s.}^F$ and $\Cl(\bG)_{s.r.s}^F$ the set of $F$-stable conjugacy classes of semisimple, regular semisimple and strongly regular semisimple elements of $\bG$ respectively.

\begin{proof}[Proof of Theorem \ref{mainthm-srs}]
Keep the notation in Theorem \ref{mainthm-srs}.
Our proof is divided into three steps, first two of which use some arguments in \cite[Chapter 3]{Car93}.
Set $W=N_{\bG}(\bT)/\bT$.
Denote by $\bT/W$ the set of $W$-orbits on $\bT$.
By \cite[3.7.2]{Car93}, $\Cl(\bG)_s^F$ is in bijection with the set $(\bT/W)^F$ of $F$-stable $W$-orbits on $\bT$,
and by \cite[Theorem 3.7.6(i)]{Car93}, $|\Cl(\bG)_s^F|=|Z(\bG)^{\circ F}|q^l$.

\emph{Step 1.}
We first show that
\[ \lim_{q\to\infty}\frac{|Z(\bG)^{\circ F}|q^l}{|\Cl(\bG)_{r.s.}^F|}=1. \]
Set
\[ A= \Set{ t\in\bT \mid \textrm{$F(t)=t^w$ for some $w\in W$ and $\alpha(t)=1$ for some $\alpha\in R$}}. \]
For any $\alpha\in R$ and $w\in W$,
denote by $\alpha^{\grp{w\phi}}$ the $\grp{w\phi}$-orbit of $\alpha$ in $X(\bT)$.
Then we claim that $A= \cup_{w\in W} \cup_{\alpha\in R} \left(\cap_{\beta\in\alpha^{\grp{w\phi}}}\Ker\beta\right)^{wF}$.
First, assume $t\in A$, then there is $w\in W$ and $\alpha\in R$ such that $\lc{w}F(t)=t$ and $\alpha(t)=1$.
In particular, $t\in\Ker\alpha$.
Then $\alpha^{wF}(t)=\alpha(\lc{w}F(t))=\alpha(t)=1$.
Since $F$ acts on $X(\bT)$ as $q\phi$, $(\alpha^{w\phi}(t))^q=1$ and thus $\alpha^{w\phi}(t)=1$.
So $t\in \Ker (\alpha^{w\phi})$.
On the other hand, the same argument shows that $\cap_{\beta\in\alpha^{\grp{w\phi}}}\Ker\beta$ is stable under the action of $wF$.

Note that $W$ acts on $A$ and $F$-stable conjugacy classes of non-regular semisimple elements of $\bG$ are in bijection with $W$-orbits on $A$.
Then it suffices to prove that
\[ \lim_{q\to\infty}\frac{|A/W|}{|Z(\bG)^{\circ F}|q^l}=0. \]
As in \cite[3.7.4]{Car93}, the number of $W$-orbits on $A$ is
\begin{align*}
&\sum_{t\in A} \frac{|W_t|}{|W|} = \frac{1}{|W|} \sum_{t\in A} \sum_{w\in W, t^w=t} 1
= \frac{1}{|W|} \sum_{t\in A} \sum_{w\in W, t^w=F(t)} 1 \\
= &\frac{1}{|W|} \sum_{w\in W} \sum_{t\in A, t^w=F(t)} 1
\leq \frac{1}{|W|} \sum_{w\in W} \sum_{\alpha\in R}
\left| \left( \cap_{\beta\in\alpha^{\grp{w\phi}}}\Ker\beta \right)^{wF} \right|.
\end{align*}
Now, we estimate $\left| \left( \cap_{\beta\in\alpha^{\grp{w\phi}}}\Ker\beta \right)^{wF} \right|$.
Set $\bS=\cap_{\beta\in\alpha^{\grp{w\phi}}}\Ker\beta$ and $\bS_0=\bS\cap[\bG,\bG]$.
Since $\bT=Z(\bG)^\circ(\bT\cap[\bG,\bG])$ and $Z(\bG)^\circ \subseteq \bS$, $\bS=Z(\bG)^\circ\bS_0$.
Let $L(\alpha)$ be the subgroup of $X(\bT)$ generated by $\alpha^{\grp{w\phi}}$.
Then by \cite[(1.7)]{Bon06}, $X(\bS/\bS^\circ) \cong (X(\bT)/L(\alpha))_{p'}$.
So $|\bS/\bS^\circ|$ is bounded by the root datum $(X,R,Y,R^\vee)$ and independent of $q$.
By Lang-Steinberg Theorem, $\bS^{wF}/\bS^{\circ wF} \cong (\bS/\bS^\circ)^{wF}$.
Thus $|\bS^{wF}/\bS^{\circ wF}|$ is bounded by the root datum $(X,R,Y,R^\vee)$ and independent of $q$.
By \cite[Proposition 4.4.9]{DM20},
\[ |\bS_0^{\circ wF}| = \left|\det\nolimits_{X(\bS_0^\circ)} (wF-1)\right|
= \left|\det\nolimits_{X(\bS_0^\circ)} (q-(w\phi)^{-1})\right|; \]
here note that $\det(w)$ and $\det(\phi)$ are signs.
Note that $\bS^\circ=Z(\bG)^\circ\bS_0^\circ$ since the dimensions of these two tori are equal,
then by the arguments in the proof of \cite[Proposition 3.3.7]{Car93},
$|\bS^{\circ wF}| = |(Z(\bG)^{\circ F}| |\bS_0^{\circ wF}|$.
Thus we have that
\[ |\bS^{wF}| \leq c\left| (Z(\bG)^{\circ F} \right| \left|\det\nolimits_{X(\bS_0^\circ)} (q-(w\phi)^{-1})\right| \]
for some constant $c$ determined by the root datum $(X,R,Y,R^\vee)$ and independent of $q$.
Since the rank of $\bS_0^\circ$ is strictly less than $l$, $|\bS^{wF}| \leq c\left| (Z(\bG)^\circ)^F \right| f(q)$,
where $f(q)$ is a polynomial in $q$ of degree strictly less than $l$ with coefficients bounded by the root datum $(X,R,Y,R^\vee)$ and $\phi$.
Thus our assertion follows.

\emph{Step 2.}
Next we use the same arguments again to show that
\[ \lim_{q\to\infty}\frac{|Z(\bG)^{\circ F}|q^l}{|\Cl(\bG)_{s.r.s.}^F|}=1. \]
Set
\[ B= \Set{ t\in\bT \mid \textrm{$F(t)=t^w$ for some $w\in W$ and $\lc{w'}t=t$ for some $w'\in W$}}. \]
For any $w'\in W$ and $w\in W$, we denote by $w'^{\grp{w\phi}}$ the $\grp{w\phi}$-orbit of $w'$.
Since $F$ and $\phi$ act on $W$ in the same way, we can show as in \emph{step 1.} that $B= \cup_{w\in W} \cup_{w'\in W} \left(\cap_{w''\in w'^{\grp{w\phi}}}\bT^{w''}\right)^{wF}$.
Note that $W$ acts on $B$ and $F$-stable conjugacy classes of non strongly regular semisimple elements in $\bG$ are in bijection with $W$-orbits on $A \cup B$.
Then by \emph{Step 1.} it suffices to prove that
\[ \lim_{q\to\infty}\frac{|B/W|}{|Z(\bG)^{\circ F}|q^l}=0. \]
As before, the number of $W$-orbits on $B$ is
\[ \sum_{t\in A} \frac{|W_t|}{|W|} \leq
\frac{1}{|W|} \sum_{w\in W} \sum_{w'\in W} \left| \left( \cap_{w''\in w'^{\grp{w\phi}}}\bT^{w''} \right)^{wF} \right|. \]
Then we estimate $\left| \left( \cap_{w''\in w'^{\grp{w\phi}}}\bT^{w''} \right)^{wF} \right|$.
Set $\bS=\cap_{w''\in w'^{\grp{w\phi}}}\bT^{w''}$ and
\[ \bS^\perp = \set{ \chi\in X(\bT) \mid \chi(x)=1, \forall x\in\bS }. \]
So $X(\bS)=X(\bT)/\bS^\perp$ and $\bS=\bS^{\perp\perp}= \set{ t\in\bT \mid \chi(t)=1, \forall \chi\in\bS^\perp }$; see for example \cite[\S1.12]{Car93}.
Note that the number of such $\bS$ to be considered is finite and depends only on the root datum $(X,R,Y,R^\vee)$.
So $|\bS/\bS^\circ|=|X(\bT)/\bS^\perp|_{tor}=|X(\bT)/\bS^\perp|_{p'}$ is bounded by the root datum $(X,R,Y,R^\vee)$ and independent of $q$.
Then with $L(\alpha)$ replaced by $\bS^\perp$, the assertion follows as in \emph{Step 1}. 

\emph{Step 3.}
Now, since $C_{\bG}(s)$ is connected for any strongly regular semisimple element $s$ of $\bG$,
the set $\Cl(\bG)_{s.r.s.}^F$ corresponds bijectively to the set $\Cl(\bG^F)_{s.r.s.}$.
For any non strongly regular semisimple element $s\in\bG^F$, the number of $\bG^F$-conjugacy classes in the $\bG$-conjugacy class of $s$ is $|H^1(F,C_\bG(s)/C_\bG(s)^\circ)|$ (see for example \cite[Proposition 4.2.14]{DM20}), thus this number is not greater than $|C_\bG(s)/C_\bG(s)^\circ|$ and is bounded by the root datum $(X,R,Y,R^\vee)$ and independent of $q$.
Then Theorem \ref{mainthm-srs} follows from \cite[Theorem 3.7.6(i)]{Car93} and the above two steps.
\end{proof}

\begin{proof}[Proof of Theorem \ref{mainthm-cl}]
Keep the notation in Theorem \ref{mainthm-cl}.
By Jordan decomposition of elements, any $x\in\bG^F$ is of the form $x=su$ with $s$ semisimple and $u\in (C_{\bG}(s)^\circ)^F$; for regular semisimple element $s$ of $\bG^F$, $u$ can only be $1$.
For non-regular semisimple element $s$ of $\bG^F$, the number of unipotent classes in $(C_{\bG}(s)^\circ)^F$ is bounded by the root datum $(X,R,Y,R^\vee)$ and independent of $q$.
Then Theorem \ref{mainthm-cl} follows from Theorem \ref{mainthm-srs}.
\end{proof}

To prove Theorem \ref{mainthm}, we need some preparations.
Let $\bG$ be a connected reductive group and $\tau: \bG_{sc} \to [\bG,\bG]$ be a simply connected covering compatible with the Frobenius maps (denoted both as $F$) on $\bG$ and $\bG_{sc}$ defining $\bbF_q$-structures.
Denote by $\bG^*$ the dual group of $\bG$ and again by $F$ the corresponding Frobenius map on $\bG^*$.
By \cite[(8.19)]{CE04}, there is an isomorphism
\[ Z(\bG^*)^F \to \Irr(\bG^F/\tau(\bG_{sc}^F)), \ z \mapsto \hat{z}. \]
By \cite[Theorem 24.17]{MT11}, if $q$ is large enough, $\bG_{sc}^F$ is perfect, and thus $\tau(\bG_{sc}^F)=[\bG^F,\bG^F]$.
Then there is an isomorphism
\begin{equation}\label{equ-linear}
Z(\bG^*)^F \to \Irr(\bG^F/[\bG^F,\bG^F]),\ z \mapsto \hat{z},\quad \textrm{when $q$ is large enuough}.
\end{equation}

\begin{lem}\label{lem-mult-free}
Let $\bG$ be a connected reductive group with a Frobenius map defining an $\bbF_q$-structure on $\bG$.
Assume $q$ is large enough.
Then $\Res^{\bG^F}_{[\bG^F,\bG^F]}$ is multiplicity free.
\end{lem}

\begin{proof}
By \cite{Lu88}, $\Res^{\bG^F}_{[\bG,\bG]^F}$ is multiplicity free.
Note also that $[\bG_0^F,\bG_0^F] \subseteq [\bG^F,\bG^F]$ with $\bG_0=[\bG,\bG]$
(in fact both derived groups are equal to $\tau(\bG_{sc}^F)$ for a simply connected covering $\tau: \bG_{sc} \to \bG_0$ when $q$ is assumed to be large enough).
Then we may assume that $\bG$ is semisimple.
Since $\bG^F$ is isomorphic to a central product of some $\bG_i^{F_i}$ with $\bG_i$ a simple algebraic group and $F_i$ a Frobenius map on $\bG_i$, we may assume furthermore that $\bG$ is a simple algebraic group (in fact we may even assume $\bG$ is of type $D$).
Let $\tau: \bG_{sc} \to \bG$ be a simply connected covering.
Since $q$ is large enough, $\bG_{sc}^F$ is perfect, thus $\tau(\bG_{sc}^F)=[\bG^F,\bG^F]$ as above.
If $\bG$ is not of adjoint type, $\bG^F/[\bG^F,\bG^F]$ is cyclic and the assertion obviously holds.
Now, assume $\bG$ is of adjoint type.
Let $\tbG$ be a regular embedding of $\bG_{sc}$, then the map $\tau$ can be extended to a surjective map $\tilde{\tau}: \tbG \to \bG$ with $\tbG$ a regular embedding of $G$.
In particular, $\Ker\tilde{\tau} = Z(\tbG)$ is connected and thus $\bG^F=\tilde{\tau}(\tbG^F)$.
Since $\tau(\bG_{sc}^F)=[\bG^F,\bG^F]$ as $q$ is large enuough,
the assertion follows from that $\Res^{\tbG^F}_{\bG_{sc}^F}$ is multiplicity free.
\end{proof}

\begin{lem}\label{lem-srs0}
Let $\bbG=(X,R,Y,R^\vee,W\phi)$ be a generic group and denote by $l$ the rank of the root datum $(X,R,Y,R^\vee)$.
Denote by $(\bG,\bT,F)$ the triple determined by $\bbG$ and $q$, and set $\bbG(q)=\bG^F$.
Denote by $\Cl(\bbG(q))_{s.r.s}^0$ the subset of $\Cl(\bbG(q))_{s.r.s}$ consisting of conjugacy classes of elements $s$ satisfying that $zs$ and $s$ are not conjugate in $\bbG(q)$ for any nontrivial $z\in Z(\bbG(q))$.
Then we have that
\[ \lim_{q\to\infty}\frac{|Z(\bG)^{\circ F}|q^l}{|\Cl(\bbG(q))_{s.r.s}^0|}=1. \]
\end{lem}

\begin{proof}
Set $\bG_0=[\bG,\bG]$.
Note that if $zs$ and $s$ are conjugate in $\bbG(q)$ for some $z\in Z(\bbG(q))$, then in fact $z\in Z(\bG_0)$.
We may assume $s\in\bT$, thus $zs\in\bT$.
Let $W$ be the Weyl group of $\bG$ with respect to $\bT$.
Then by \cite[3.7.1]{Car93}, there is $w\in W$ such that $\lc{w}s=zs$.
Note also that each $\bG$-conjugacy class of strongly regular semisimple elements corresponds to a unique $\bG^F$-conjugacy class.

Let $C$ be the set of $t\in\bT$ such that $F(t)=t^w$ for some $w\in W$ and $\lc{w'}t=zt$ for some $w'\in W$ and $z\in Z(\bG_0)^F$.
Then the set $B$ in the proof of Theorem \ref{mainthm-srs} is a subset of $C$.
For any $w'\in W$, denote by $\bT^{(w')}$ the diagonalizable subgroup of $\bT$ of elements $t\in\bT$ such that $\lc{w'}t=zt$ for some $z\in Z(\bG_0)^F$.
Note that $|Z(\bG_0)|$ is finite (and bounded by the root datum $(X,R,Y,R^\vee)$ and independent of $q$).
Thus $\bT^{(w')}$ has dimension less than that of $\bT$.

As in \emph{step 2.} of the proof of Theorem \ref{mainthm-srs},
we have $C= \cup_{w\in W} \cup_{w'\in W} \left(\cap_{w''\in w'^{\grp{w\phi}}}\bT^{(w'')}\right)^{wF}$.
Then it suffices to show that
\[ \lim_{q\to\infty}\frac{|C/W|}{|Z(\bG)^{\circ F}|q^l}=0. \]
This can be proved by the same argument of \emph{step 2.} in the proof of Theorem \ref{mainthm-srs}.
\end{proof}

\begin{proof}[Proof of Theorem \ref{mainthm}]
Keep the notation in Theorem \ref{mainthm}.
Assume $\iota: \bG \to \tbG$ is a regular embedding and denote by $\iota^*: \tbG^*\to\bG^*$ the corresponding dual map.
For convenience, we denote by $F$ the Frobenius map on all these reductive groups.
By Lusztig's theory, $\Irr(\bG^F)$ is a union of Lusztig series $\cE(\bG^F,s)$ with $s$ running over a set of representatives of conjugacy classes of semisimple elements of $\bG^{*F}$.
By the proof of \cite{Lu88}, the characters in $\cE(\bG^F,s)$ are irreducible constituents of characters in $\cE(\tbG^F,\ts)$ with $s=\iota^*(\ts)$ for some $\ts\in\tbG^*$.
When $s$ is a strongly regular semisimple element of $\bG^{*F}$, $C_{\bG^*}(s)$ is a maximal torus,
thus by \cite{Lu88} (see also \cite[15.14]{CE04}),
$\cE(\bG^F,s)$ is a single point set whose unique element is denoted as $\chi_s$.
By Lemma \ref{lem-srs0} and Theorem \ref{mainthm-cl},
it suffices to prove that $\Res^{\bG^F}_{[\bG^F,\bG^F]}\chi_s$ is irreducible for any $s\in(\bG^{*F})_{s.r.s.}^0$,
where the meaning of $(\bG^{*F})_{s.r.s.}^0$ is as in Lemma \ref{lem-srs0}.

When $q$ is large enough, we have an isomorphism (\ref{equ-linear}).
Let $\chi\in\Irr(\bG^F)$, then by Lemma \ref{lem-mult-free} and Clifford theory, $\Res^{\bG^F}_{[\bG^F,\bG^F]}\chi$ is irreducible if and only if $\hat{z}\chi\neq\chi$ for any nontrivial $z\in Z(\bG^*)^F$.
By \cite[(8.20)]{CE04}, $\hat{z}\chi_s=\chi_{zs}$ for the strongly regular semisimple element $s$ of $\bG^{*F}$.
Thus $\hat{z}\chi_s=\chi_s$ if and only if $zs$ and $s$ are $\bG^{*F}$-conjugate.
But $s\in(\bG^{*F})_{s.r.s.}^0$, so $zs$ and $s$ are not $\bG^{*F}$-conjugate for any nontrivial $z\in Z(\bG^*)^F$.
So $\hat{z}\chi_s\neq\chi_s$ for any nontrivial $z\in Z(\bG^*)^F$ and thus $\Res^{\bG^F}_{[\bG^F,\bG^F]}\chi_s$ is irreducible as required.
\end{proof}

\begin{rem}
When $Z(\bG)$ is connected and $\bG_0=[\bG,\bG]$ is of simply connected, we give an explanation for the proofs of Lemma \ref{lem-srs0} and Theorem \ref{mainthm}.
Set $\bG_0=[\bG,\bG]$ and denote by $\iota: \bG_0 \to \bG$ the natural embedding.
Denote by $\bG^*,\bG_0^*$ the dual groups of $\bG,\bG_0$ respectively.
Let $\iota^*: \bG^* \to \bG_0^*$ the dual of $\iota$.
For convenience, we denote by $F$ the Frobenius map on all these groups compatible with $\iota$, $\iota^*$ and duality.
Then by \cite[Lemme 8.3]{Bon06}, the conjugacy class of $s\in\bG^{*F}$ is in $\Cl(\bG^{*F})_{s.r.s.}^0$ if and only if $\iota^*(s)$ is a strongly regular semisimple element of $\bG_0^{*F}$.
Thus $\cE(\bG_0,\iota^*(s))$ is a single set whose unique element $\chi_{\iota^*(s)}$ can be extended to $\bG$ and all characters of $\Irr(\bG^F)_{i.r.d.}$ can be obtained in this way.
\end{rem}

\section{An example: general linear and unitary groups}

We give a different direct method for the general linear groups and general unitary groups.
Let $\GL_n(-q)$ denote $\GU_n(q)$ and $\SL_n(-q)$ denotes $\SU_n(q)$.
Set $G_n(q)=\GL_n(\epsilon{q})$ with $\epsilon=\pm1$, then $[G_n(q),G_n(q)]=\SL_n(\epsilon{q})$.

We first recall a parametrization of irreducible characters of $G_n(q)$.
For an arbitrary field $k$,
denote by $k[X]$ ($\Irr(k[X])$, resp.) the set of all polynomials (all monic irreducible polynomials, resp.) over $k$.
For $\Delta(X)=X^m+a_{m-1}X^{m-1}+\cdots+a_0$ in $\bbF_{q^2}[X]$,
define $\tDelta(X)=X^ma_0^{-q}\Delta^q(X^{-1})$,
where $\Delta^q(X)$ denotes the polynomial whose coefficients are the $q$-th powers of the corresponding coefficients of $\Delta(X)$.
Set
\begin{align*}
\cF_0 &= \Set{ \Delta ~\middle|~ \Delta\in\Irr(\bbF_q[X]),\Delta\neq X }, \\
\cF_1 &= \Set{ \Delta ~\middle|~ \Delta\in\Irr(\bbF_{q^2}[X]),\Delta\neq X,\Delta=\tDelta }, \\
\cF_2 &= \Set{ \Delta\tDelta ~\middle|~ \Delta\in\Irr(\bbF_{q^2}[X]),\Delta\neq X,\Delta\neq\tDelta },
\end{align*}
and $\cF=\cF_0$ or $\cF_1\cup\cF_2$ according to $\epsilon=1$ or $-1$.
Denote the map $\barF_q^\times\to\barF_q^\times$, $\alpha \to \alpha^{\epsilon q}$ as $F_{\epsilon q}$.
Then any polynomial $\Gamma\in\cF$ can be identified with an orbit of $\grp{F_{\epsilon q}}$ on $\barF_q^\times$.
For any semisimple element $s$ in $G_n(q)$,
denote by $m_\Gamma(s)$ the multiplicity of $\Gamma$ as an elementary divisor of $s$.
By Lusztig's Jordan decomposition of characters,
irreducible characters of $G_n(q)$ can be parameterized by $G_n(q)$-conjugacy classes of pairs $(s,\lambda)$,
where $s$ is a semisimple element of $G_n(q)$ and $\lambda=\prod_\Gamma\lambda_\Gamma$ with $\lambda_\Gamma$ a partition of $m_\Gamma(s)$.
The character of $G_n(q)$ corresponding to $(s,\lambda)$ is denoted by $\chi_{s,\lambda}$.

Denote by $\Lin(G_n(q))=\Irr(G_n(q)/[G_n(q),G_n(q)])$ the set of all linear characters of $G_n(q)$.
From Clifford theory of irreducible characters and the fact that $G_n(q)/[G_n(q),G_n(q)]$ is cyclic,
we have for any irreducible character $\chi$ of $G_n(q)$ that
\[ |\Irr([G_n(q),G_n(q)]\mid\chi)|= |\set{ \eta\in\Lin\left(G_n(q)\right) \mid \chi\eta=\chi }|. \]
On the other hand, there is an isomorphism
\[ Z(G_n(q)) \to \Lin(G_n(q)),\quad z\mapsto\hz. \]
We will always dentify $Z(G_n(q))$ with the set $\set{ z\in\barF_q^\times \mid z^{q-\epsilon}=1 }$.
For any $z\in Z(G_n(q))$ and $\Gamma\in\cF$,
denote by $z\Gamma$ the polynomial in $\cF$ whose roots are $z\alpha$ with $\alpha$ running through all roots of $\Gamma$.
Then $m_{z\Gamma}(zs)=m_\Gamma(s)$.
By \cite[Theorem 4.7.1 (3)]{GM20}, the above isomorphism can be chosen such that
\[ \hz\chi_{s,\lambda}=\chi_{zs,z\lambda}, \]
where $z\lambda$ is defined as $(z\lambda)_{z\Gamma}=\lambda_\Gamma$.

Denote by $\Irr_k(G_n(q))$ the set of irreducible characters of $G_n(q)$ stable under the multiplication of the subgroup of $Z(G_n(q)) \cong \Lin(G_n(q))$ of order $k$.
Let $c_{n,k}(q)=|\Irr_k(G_n(q))|$.
In particular, $c_n(q):=c_{n,1}(q)=|\Irr(G_n(q))|$.
Fix a generator $z_k$ of the subgroup of $Z(G_n(q))$ of order $k$.
We will give an estimate for $c_{n,k}(q)$ by the same argument in \cite{Mac81}.

Denote by $\cP(\bbN)$ the set of all partitions of natural numbers (including the empty partition of $0$ for convenience).
Then it is easy to see that the set $\Irr_k(G_n(q))$ is in bijection with the set of all partition-valued maps $\mu:\barF_q^\times\to\cP(\bbN)$ satisfying the following conditions:
\[ \sum_{\zeta\in\barF_q^\times}|\mu(\zeta)|=n,\quad \mu(\zeta^{\epsilon q})=\mu(\zeta)=\mu(z_k\zeta). \]

For any partition $\lambda$, we denote by $m_i(\lambda)$ the multiplicity of $i$ appearing in $\lambda$.
For any $\mu$ as above, let
\[ u_i(X)=\prod_{\zeta\in\barF_q^\times}(1-\zeta X)^{m_i(\mu(\zeta))}. \]
The polynomial $u_i\in\barF_q[X]$ satisfies the following:
\begin{equation}\label{equ-condition-u_i}
u_i(0)=1,\quad u_i(z_kX)=u_i(X),\quad u_i\in\tcF,
\end{equation}
where $\tcF=\bbF_q[X]$ for $\epsilon=1$ while $\tcF$ is the set of all polynomials $\Delta$ in $\bbF_{q^2}[X]$ such that $\alpha^{-q}$ is a root of $\Delta$ whenever $\alpha$ is a root of $\Delta$ for $\epsilon=-1$.
Then there is a bijection between the set $\Irr_k(G_n(q))$ and the set of sequences $u=(u_1,u_2,\dots)$ of polynomials with each $u_i$ satisfying (\ref{equ-condition-u_i}) and $\sum\limits_{i\geqslant1}i\deg u_i=n$.
In fact, for an element $g$ in the conjugacy class corresponding to $\mu$, we have
\[ \det(I_n-gX)=\prod_{i\geqslant1}u_i(X)^i. \]

Let $u=(u_1,u_2,\dots)$ be as above.
Denote $n_i=\deg u_i$ and $\nu=(1^{n_1}2^{n_2}\cdots)$, then $\nu$ is a partition of $n$.
We call $\nu$ the type of the irreducible character corresponding to $u$.
Since the polynomials $u_i(X)$ satisfying $u_i(z_kX)=u_i(X)$ are exactly those polynomials whose monomials are of the form $aX^{kj}$,
the number of polynomials $u_i$ of degree $n_i$ satisfying (\ref{equ-condition-u_i}) is:
\[ \begin{cases}
0, & \textrm{if $k\nmid n_i$}; \\
1, & \textrm{if $n_i=0$}; \\
q^{\frac{n_i}{k}}-\epsilon q^{\frac{n_i}{k}-1}, & \textrm{if $k\mid n_i>0$}.
\end{cases} \]
Consequently, the number $c_{\nu,k}(q)$ of conjugacy classes in $C_{n,k}(q)$ of type $\nu=(1^{n_1}2^{n_2}\cdots)$ is
\[ c_{\nu,k}(q)=\begin{cases}
\prod\limits_{n_i>0}\left(q^{\frac{n_i}{k}}-\epsilon q^{\frac{n_i}{k}-1}\right),
& \textrm{if}~k\mid{n_i}~\textrm{for any}~i;\\
0, & \textrm{otherwise}.
\end{cases} \]
So we have that
\[ c_{n,k}(q)
=\sum\limits_{|\nu|=n}c_{\nu,k}(q)
=\sum\limits_{\small\begin{array}{c}|\nu|=n, k\mid n_i,\forall i \end{array}}
\prod\limits_{n_i>0}\left(q^{\frac{n_i}{k}}-\epsilon q^{\frac{n_i}{k}-1}\right). \]
In particular,
\begin{compactenum}[(1)]
\item
if $k\nmid n$, we have $c_{n,k}(q)=0$;
\item
if $k\mid n$, we have $c_{n,k}(q)=q^{\frac{n}{k}}+f(q)$ with $f(X)\in\bbZ[X]$ and $\deg f(X)<\frac{n}{k}$.
\end{compactenum}

By the above results, $c_{n,k}(q)=0$ if $k\nmid(n,q-\epsilon)$;
while if $k\mid(n,q-\epsilon)$, we have
\[ c_{n,k}(q)= O(q^{\frac{n}{k}}),\quad q\to\infty. \]
In particular, we have that $|\Irr(G_n(q))|=c_n(q)= O(q^n),\ q\to\infty$.
Denote by $\Irr_r(G_n(q))$ the set of irreducible characters of $G_n(q)$ whose restrictions to $[G_n(q),G_n(q)]$ are reducible.
Then if $(n,q-\epsilon)=1$, $|\Irr_r(G_n(q))|=0$;
while if $(n,q-\epsilon)>1$, we have
\[ |\Irr_r(G_n(q))|= O(q^{\frac{n}{\ell}}),\quad q\to\infty, \]
where $\ell$ is the minimal prime dividing $(n,q-\epsilon)$.
Thus Theorem \ref{mainthm} for $G_n(q)=\GL_n(\epsilon q)$ follows.

Now consider a special case $\GL_\ell(\epsilon{q})$,
where $\epsilon=\pm1$, $\GL_\ell(-q)$ denotes $\GU_\ell(q)$ and $\ell$ divides $q-\epsilon$.
\begin{compactenum}[(i)]
\item
There are exactly $q-\epsilon$ irreducible characters of $\GL_\ell(\epsilon q)$ whose restrictions to $[\GL_\ell(\epsilon q),\GL_\ell(\epsilon q)]=\SL_\ell(q)$ are not irreducible.
\item
Under the action of $(\GL_\ell(\epsilon q))$,
these $q-\epsilon$ irreducible characters form $\ell$ orbits,
each of which contains $(q-\epsilon)/\ell$ irreducible characters;
and all characters in each orbit have the same restriction to $\SL_\ell(\epsilon q))$,
which is a sum of $\ell$ irreducible characters of $\SL_\ell(\epsilon q))$.
\item
There are exactly $\ell^2$ irreducible characters of $\SL_\ell(\epsilon q))$ which can not be extended to $\GL_\ell(\epsilon q))$;
this number is independent of $q$.
\end{compactenum}

A byproduct is a generalization of the generating function $c(t)=\sum\limits_{n=0}\limits^{\infty}c_n(q)t^n$ of $c_n(q)$ in \cite{Mac81} with $t$ an indeterminant.
By \cite{FF60} and \cite{Mac81}, $c(t)=\prod\limits_{r=1}\limits^{\infty}\frac{1-\epsilon t^r}{1-qt^r}$.
Let $c_k(t)=\sum\limits_{n=0}\limits^{\infty}c_{n,k}(q)t^n$ be the generating function of $c_{n,k}(q)$.
Then by the similar argument in \cite{Mac81}, we have that
\[ c_k(t)=\prod\limits_{r=1}\limits^{\infty}\frac{1-\epsilon t^{kr}}{1-qt^{kr}}. \]


\end{document}